\documentclass{article}
\usepackage{amsmath}
\usepackage{amssymb}
\usepackage{amsthm}
\usepackage[shortlabels]{enumitem}

\newcommand\E{\mathbb{E}}

\newtheorem{thm}{Theorem}
\newtheorem{lem}[thm]{Lemma}
\newtheorem{cor}[thm]{Corollary}

\title{Extension of a method of Gilmer}
\date{Wednesday 23, November}
\author{Luke Pebody}
\begin{document}
\maketitle
\begin{abstract}
It is a well-known conjecture, sometimes attributed to Frankl, that for any family of sets which is 
closed under the union operation, there is some element which is contained in at least half of the
sets. 

Gilmer~\cite{Gilmer} was the first to prove a constant bound, showing that there is some element
contained in at least 1\% of the sets. They state in their paper that the best possible bound
achievable by the same method is $\frac{3-\sqrt5}2\approx 38.1\%$.

This note achieves that bound by finding the optimum value, given a binary variable $X$ 
potentially depending on some other variable $S$ with a given expected value $E(X)$ and
conditional entropy $H(X|S)$ of the conditional entropy of $H(X_1\cup X_2|S_1,S_2)$ for
independent readings $X_1, S_1$ and $X_2,S_2$.
\end{abstract}
\section{Introduction}
In this note, we prove the following result (strengthening almost as much 
as possible the main result from~\cite{Gilmer}).
\begin{thm}\label{T:main}
Let $A$ and $B$ denote independent samples from a distribution over subsets of $[n]$. Assume that for
all $i\in[n]$, $\Pr[i\in A]\le\alpha$ for some value $0<\alpha\le\frac{3-\sqrt{5}}2$. Then
$H(A\cup B)\ge\frac{H(\alpha^2)}{H(\alpha)}H(A).$
\end{thm}

This will imply the following.
\begin{thm}
Let $\mathcal{F}\subseteq2^{[n]}$ be a union-closed family, $\mathcal{F}\ne\emptyset$. Then there exists 
$i\in[n]$ that is contained in at least a $\frac{3-\sqrt{5}}{2}\approx 38.1\%$ fraction of the sets 
in $\mathcal{F}$.
\end{thm}

\begin{proof}
Suppose not. Let $A$ and $B$ be drawn independently and uniformly from $\mathcal{F}$. Then, since
$\Pr[i\in A]$ is rational for each $i$, there exists some $\alpha<\frac{3-\sqrt{5}}2$ for which
$\Pr[i\in A]\le\alpha$ for all $i\in [n]$. Then from Theorem~\ref{T:main} it follows that 
$H(A\cup B)\ge\frac{H(\alpha^2)}{H(\alpha)}H(A)>H(A)$. But this causes a contradiction, as
$A$ is drawn from the uniform distribution on $\mathcal{F}$ and $A\cup B$ is drawn from some other
distribution on $\mathcal{F}$ and the uniform distribution has the maximum entropy.
\end{proof}

\section{Main Lemma}
\begin{lem}\label{L:main}
For $1\le i\le n$, let $p_i, v_i\in[0,1]$ be numbers such that $\sum_i p_i=1$ and 
$\sum_i p_iv_i\le\alpha$ for some $0<\alpha\le\frac{3-\sqrt{5}}2$. 
Then \[\sum_{i,j}p_ip_jH(v_i+v_j-v_iv_j)\ge\frac{H(2\alpha-\alpha^2)}{H(\alpha)}\sum_ip_iH(v_i).\]
\end{lem}

\begin{proof}[Proof of Theorem~\ref{T:main} from Lemma~\ref{L:main}]
The proof follows exactly as the Proof of Theorem 1 from Lemma 1 in~\cite{Gilmer}.
\end{proof}

For ease of writing in the remainder, we use the fact that $H(1-x)=H(x)$ to rewrite Lemma~\ref{L:main}
in terms of $w_i=1-v_i$.

\begin{lem}[Lemma~\ref{L:main} Rewritten]\label{L:main2}
For $1\le i\le n$, let $p_i, w_i\in[0,1]$ be numbers such that $\sum_i p_i=1$ and 
$\sum_i p_iw_i\ge\beta$ for some $1>\beta\ge\frac{\sqrt{5}-1}2$. 
Then \[\sum_{i,j}p_ip_jH(w_iw_j)\ge\frac{H(\beta^2)}{H(\beta)}\sum_ip_iH(w_i).\]
\end{lem}

\section{Some monotonic functions}
\begin{lem}\label{L:turlough}
The function $\frac{H(x^2)}{H(x)}$ is increasing on the range $[0,1]$.
\end{lem}
\begin{proof}
To be proven.
\end{proof}

\begin{lem}\label{L:adric}
The function $\frac{H(x^2)}{xH(x)}$ is increasing on the range
$[\frac{\sqrt{5}-1}2,1]$.
\end{lem}

\begin{proof}
To be proven.
\end{proof}

\section{Properties of $H(x)/x$}
Our proof depends on various properties of $\frac{H(x)}{x}$ (and its inverse). For ease of notation,
let us write $f(x)=\frac{H(x)}x$.

\begin{lem}\label{L:kamelion}
The function $f$ is continuous, onto and decreasing from $(0,1]$ to $(0,\infty)$.
\end{lem}

\begin{proof}
The derivative of $f$ is $f'(x)=\frac{\log(1-x)}{x^2}$, which is negative on the region.
Clearly $\frac{H(1)}{1}=0$ and as $x\to 0$, $\frac{H(x)}{x}>-\log x$ tends to infinity.
\end{proof}

It follows that for all non-negative real $y$, there is a unique value $g(y)$ for which $f(g(y))=y$.
By standard properties of derivatives, if $g(y)=x$, $g$ is differentiable at $y$, with derivative 
$1/f'(x).$

\begin{lem}\label{L:peri}
For all $0<\alpha<1$, the function $f(\alpha g(x))$ is convex.
\end{lem}

\begin{proof}
The derivative of $f(\alpha g(x))$ with respect to $x$ is 
\begin{align*}
f'(\alpha g(x))\alpha g'(x)&=\frac{\log(1-\alpha g(x))}{\alpha^2g(x)^2}\frac{\alpha}{f'(g(x))}\\
&=\frac{\log(1-\alpha g(x))}{\alpha^2g(x)^2}\frac{\alpha g(x)^2}{\log(1-g(x))}\\
&=\frac{\log(1-\alpha g(x))}{\alpha\log(1-g(x))}.
\end{align*}

We will show this is increasing in $x$. Since $g(x)$ is decreasing, this is the same as showing 
that $\frac{\log(1-\alpha y)}{\alpha\log(1-y)}$ is decreasing on the range $(0,1]$.

The derivative of this with respect to $y$ is
\begin{align*}
\frac{d}{dy}
&=\frac{-\alpha/(1-\alpha y)\alpha\log(1-y)+\alpha/(1-y)\log(1-\alpha y)}{\alpha^2\log(1-y)^2}\\
&=\frac{-\alpha(1-y)\log(1-y)+(1-\alpha y)\log(1-\alpha y)}{\alpha\log(1-y)^2(1-y)(1-\alpha y)^2}\\
&=\frac{-(1-y)\log(1-y)/y+(1-\alpha y)\log(1-\alpha y)/(\alpha y)}{\log(1-y)^2(1-y)(1-\alpha y)^2/y}.
\end{align*}

So showing that $f(\alpha g(x))$ is convex for all $0<\alpha<1$ is the same as showing that
$-(1-\alpha y)\log(1-\alpha y)/(\alpha y)>-(1-y)\log(1-y)/y$ for all $0<\alpha<1$, which is the
same as showing that $-(1-z)\log(1-z)/z$ is decreasing on $[0,1]$. 

An elementary property of the derivative of quotients is that if $g'(x)>0$ and $g(x)>0$, 
$f(x)/g(x)$ is decreasing whenever $f(x)/g(x)\ge f'(x)/g'(x)$. Thus $-(1-z)\log(1-z)/z$ is 
decreasing if $-(1-z)\log(1-z)/z\ge1+\log(1-z)$, which can be rewritten as 
$\log(1-z)\le -z$, or $1-z\le\exp(-z)$. Since this is true for all $z$, we are done.
\end{proof}

\section{Optimising Joint Entropy}
In this section we will show that for any distribution $p$ on $[0,1]$, there exists a
distribution $q$ on $[0,1]$ with at most one non-zero point in its support, for which
$\E_q(X)=\E_p(X)$, $\E_q(H(X))=\E_p(H(X))$ and $\E_{q,q}(H(X_1X_2))\le\E_{p,p}(H(X_1X_2)$.

We will do this by proving first for distributions with finite support, inductively by reducing 
the size of the non-zero support by one at a time.

This will be achieved by the following technical lemma.
\begin{lem}\label{L:ace}
Given numbers $p_1,p_2,x_1,x_2\in(0,1]$, there exists numbers $q$ and $y$ such that
$p_1x_1+p_2x_2=qy$ and $p_1H(x_1)+p_2H(x_2)=qH(y)$.

For such $q$ and $y$, the following properties hold:
\begin{enumerate}
\item $q\le p_1+p_2$,
\item For all $0\le z\le 1$, $p_1H(zx_1)+p_2H(zx_2)\ge qH(zy)$ and
\item $p_1^2H(x_1^2)+2p_1p_2H(x_1x_2)+p_2^2H(x_2^2)\ge q^2H(y)^2$.
\end{enumerate}
\end{lem}

\begin{proof}
Since $f:t\to H(t)/t$ is onto, there exists $y$ such that
\begin{align*}
\frac{H(y)}{y}=f(y)
&=\frac{p_1x_1}{p_1x_1+p_2x_2}f(x_1)+\frac{p_2x_2}{p_1x_1+p_2x_2}f(x_2)\\
&=\frac{p_1H(x_1)+p_2h(x_2)}{p_1x_1+p_2x_2}.
\end{align*}

Thus if we set $q=\frac{p_1x_1+p_2x_2}y$, the required equalities clearly hold.

By construction, $\frac{H(y)}{y}$ is contained in the interval spanned by
$\frac{H(x_1)}{x_1}$ and $\frac{H(x_2)}{x_2}$. Since $f$ is monotonic, it follows that
$y$ is contained in the interval spanned by $x_1$ and $x_2$. Hence it follows that $q\le p_1+p_2$.

Fix $0\le z\le 1$. By Lemma~\ref{L:peri}, $f(z g(x))$ is convex. If we write 
$p=\frac{p_1x_1}{p_1x_1+p_2x_2}$ then since $f(y)=pf(x_1)+(1-p)f(x_2)$, it follows that
$f(z g(f(y)))\le pf(zg(f(x_1)))+(1-p)f(\alpha g(f(x_2)))$. Since $g$ and $f$ are inverse 
functions, this can be rewritten as $f(zy)\le pf(zx_1)+(1-p)f(zx_2)$, which gives
\begin{align*}
\frac{H(zy)}{y}&\le\frac{p_1x_1}{p_1x_1+p_2x_2}\frac{H(zx_1)}{x_1}+
\frac{p_2x_2}{p_1x_1+p_2x_2}\frac{H(zx_2)}{x_2}\\
H(zy)&\le p_1x_1\frac{H(zx_1)}{x_1}+p_2x_2\frac{H(zx_2)}{x_2}\\
&=p_1H(zx_1)+p_2H(zx_2).
\end{align*}

Finally the last inequality comes from the second inequality applied at $z=y, x_1$ and $x_2$:
\begin{align*}
q^2H(y^2)&\le q(p_1H(yx_1)+p_2H(yx_2))\\
&=p_1(qH(yx_1))+p_2(qH(yx_2))\\
&\le p_1(p_1H(x_1)^2+p_2H(x_1x_2))+p_2(p_1H(x_1x_2)+p_2H(x_2)^2).
\end{align*}
\end{proof}

This solves the above optimisation problem.
\begin{thm}\label{T:tegan}
For any distribution $p$ on $[0,1]$, there exists a
distribution $q$ on $[0,1]$ with at most one non-zero point in its support, for which
$\E_q(X)=\E_p(X)$, $\E_q(H(X))=\E_p(H(X))$ and $\E_{q,q}(H(X_1X_2))\le\E_{p,p}(H(X_1X_2)$.
\end{thm}

\begin{proof}
Suppose first that $p$ is of finite support, but has more than one non-zero point in its support.
Let $x_1, x_2$ be distinct non-zero points in the support and let $p_1=\Pr(X=x_1)$ and $p_2=\Pr(X=x_2)$.

If we let $q$ and $y$ be the numbers described in Lemma~\ref{L:ace}, and make $p'$ be a distribution
which is identical to $p$ but replaces hitting the elements $x_1, x_2$ with probability $p_1, p_2$ by
hitting the elements $y, 0$ with probability $q, p_1+p_2-q$, we still have a probability 
distribution (as $0\le q\le p_1+p_2$), but with smaller support. Further, by the definition of $q$ and
$y$, $\E_p(X)=\E_{p'}(X)$ and $\E_p(H(X))=\E_{p'}(H(X))$. Further, from the inequalities in
Lemma~\ref{L:ace}, $\E_{p',p'}(H(X_1X_2))\le\E_{p,p}(H(X_1X_2)$.

Inductively, for any distribution of finite support we can replace two non-zero elements with one
repeatedly, maintaining $\E_p(X)$ and $\E_p(H(X))$ but never increasing $\E_{p,p}(H(X_1X_2))$.
Eventually we will reach such a distribution $q$. This proves the Theorem for all distributions
of finite support.

The distributions of finite support are dense in the set of all distributions and these expectations
are continuous. As such, we can generalise to all distributions.
\end{proof}

We will rewrite this Theorem using the function $g$ which we defined as the inverse of $f(x)=H(x)/x$.

\begin{cor}\label{C:nyssa}
For real numbers $0<t,u<1$ with $u\le H(t)$, the minimum possible value of 
$\E_{p,p}(H(X_1X_2))$ with $\E_p(X)=t$ and $\E_p(H(x))=u$ is 
$t^2H(v^2)/v^2$ where $v=g(u/t)$.
\end{cor}

\begin{proof}
This is achievable by the distribution which takes $v$ with probability $\frac{t}{v}$ and
0 with probability $1-\frac{t}{v}$, note that $f(v)=u/t\le H(t)/t=f(t)$, so $t<v$ and this is 
a proper distribution.

To show this is optimal, take any distribution $p$. Then we know by Theorem~\ref{T:tegan}
there is a distribution $q$ on $[0,1]$ with at most one non-zero point in its support for which
$\E_q(X)=t$, $\E_q(H(X))=u$ and $\E_{q,q}(H(X_1)X_2))\le\E_{p,p}(H(X_1X_2))$.

Since the expectation of $X$ under $q$ is non-zero, there is a non-zero point in the support of 
$q$. Let it be $v'$ and the probability be $p$.

Then $t=pv'$ and $u=pH(v')$, so $f(v')=u/t=f(v)$ and hence $v=v'$ and $p=\frac{t}{v}$, so
the distribution $q$ is in fact the construction from the first paragraph of this proof.
\end{proof}

\begin{proof}[Proof of Lemma~\ref{L:main2}]
Let us suppose as defined in the lemma, that for $1\le i\le n$, $p_i$ and $w_i\in[0,1]$ are 
numbers such that $\sum_i p_i=1$ and 
$\sum_i p_iw_i\ge\beta$ where $1>\beta\ge\frac{\sqrt{5}-1}2$.

Set $t=\sum_ip_iw_i$. Since $H$ is concave, it follows that $0\le\sum_ip_iH(w_i)\le H(t)$. 

Set $u=\sum_ip_iH(w_i)$ and $v=g(u/t)$. Then from Corollary~\ref{C:nyssa} it follows that
\[\sum_i\sum_j p_ip_jH(w_iw_j)\ge t^2H(v^2)/v^2.\]

Recall that $f(v)=u/t\le H(t)/t=f(t)$ so $t<v$. From Lemma~\ref{L:adric}, 
$\frac{H(v^2)}{vH(v)}\ge\frac{H(t^2)}{tH(t)}$. Note that this is the only place in the proof we use that $\beta\ge\frac{\sqrt{5}-1}2$. 

It follows that
\begin{align*}
\sum_i\sum_j p_iH(w_iw_j)&\ge t^2H(v^2)/v^2\\
&=t\frac{H(v^2)}{vH(v)}\frac{tH(v)}{v}\\
&=tu\frac{H(v^2)}{vH(v)}\\
&\ge tu\frac{H(t^2)}{tH(t)}\\
&=\frac{H(t^2)}{H(t)}\sum_ip_ih(W_i)\\
&\ge\frac{H(\beta^2}{H(\beta)}\sum_i p_iH(w_i),
\end{align*}
the last inequality coming from Lemma~\ref{L:turlough}
\end{proof}

\bibliography{GilmerPaper}
\bibliographystyle{ieeetr}
\end{document}